\documentclass[10pt,a4paper]{article}
\title{A definability criterion for connected Lie groups}
\author{Alf Onshuus, Sacha Post \\ Universidad de los Andes, Bogot\'a Colombia}

\usepackage[utf8]{inputenc}
\usepackage{comment}
\usepackage[english]{babel}
\usepackage[T1]{fontenc}
\usepackage{textcomp}
\usepackage{amsmath}
\usepackage{mathtools}
\usepackage{amsthm, thmtools}
\usepackage{thm-restate}
\usepackage{stmaryrd}
\usepackage{amsfonts}
\usepackage{amssymb}
\usepackage{lmodern}
\usepackage{hyperref}
\usepackage{cleveref}
\usepackage{csquotes}
\usepackage{tikz-cd}
\usepackage{enumitem}
\usepackage{titlesec}

\newtheorem{theorem}{Theorem}
\newtheorem{lemma}{Lemma}[section]
\newtheorem{corollary}[lemma]{Corollary}
\newtheorem{fact}[lemma]{Fact}

\newtheorem{proposition}[lemma]{Proposition}
\newtheorem{claim}[lemma]{Claim}

\addto\extrasenglish{
	
}

\begin{document}
	
	\maketitle

	\begin{abstract}
		It has been known since \cite{Pgroupchunk} that any group definable in an $o$-minimal expansion of the real field can be equipped with a Lie group structure. It is therefore natural to ask when is a Lie group Lie isomorphic to a group definable in such an expansion. Conversano, Starchenko and the first author answered this question in \cite{COSsolvable} in the case when the group is solvable. This paper answers similar questions in more general contexts.

We first give a complete classification in the case when the group is linear. Specifically, a linear Lie group $G$ is Lie isomorphic to a group definable in an $o$-minimal expansion of the reals if and only if its solvable radical has the same property.

We then deal with the general case of a connected Lie group, although unfortunately we cannot achieve a full characterization. Assuming that a Lie group $G$ has a ``good Levi descomposition'', we prove that in order for $G$ to be Lie isomorphic to a definable group it is necessary and sufficient that its solvable radical satisfies the conditions given in \cite{COSsolvable}.
	\end{abstract}
	
	\section{Introduction}
	
	We start with some definitions. We shall follow the usual notation for classical $o$-minimal structures (see \cite{vdD}): $\mathbb{R}_{\exp}$ for the structure $(\mathbb{R},0,+,1,\cdot,<,\exp)$ and $\mathbb{R}_{an,\exp}$ for the structure $\left(\mathbb{R},0,+,1,\cdot,<,\exp,\left\{f\right\}_{f\in\mathcal{F}}\right)$ where $\mathcal{F}$ is the set of analytic functions $f:\mathbb{R} \rightarrow \mathbb{R}$ with compact support.
	
	We will always work with $\mathbb R$ as our universe. In particular, by \emph{definable group} we will mean a group definable in a cartesian power of an $o$-minimal expansion of the reals. We will say that a  group is \emph{definably linear} if it is a definable group that acts definably, faithfully and linearly on $\mathbb{R}^n$, for some integer $n$.
	
	\medskip
	
A  group definable in an $o$-minimal expansion of the reals can be equipped with a Lie group structure (see \cite{Pgroupchunk}). In this paper we work on some partial results to characterize when the converse is true. An earlier result, a characterization of when a solvable Lie group is Lie-isomorphic to a definable group was achieved by A. Conversano, S. Starchenko and the first author:

	\begin{fact}[Solvable case \cite{COSsolvable}, Theorem $5.4$] \label{solvablecriterion}
		Let $R$ be a solvable Lie group. Then the following are equivalent:
		\begin{itemize}
			\item $R$ has a normal, connected, torsion-free and supersolvable subgroup $T$ such that $R/T$ is compact (we say that $R$ is \emph{triangular by compact}).
			\item $R$ is Lie isomorphic to a group definable in an $o$-minimal expansion of the reals.
		\end{itemize}
	\end{fact}

Recall that a connected Lie group is said to be \emph{supersolvable} (sometimes called \emph{triangular}) if the eigenvalues of the operator $Ad(g)$ of the adjoint representation are real for all $g\in G$. This is equivalent to saying that $\mathfrak{g}=Lie(G)$ is supersolvable as described in \cite[Section 2]{COSsolvable}.
	
\medskip

The paper is divided as follows. In \Cref{sec2} we strengthen \Cref{solvablecriterion} to show that any Lie group satisfying the first condition in \Cref{solvablecriterion} is in fact Lie-isomorphic to a definably linear group. Namely we prove

	\begin{restatable}{theorem}{solvlintheorem}\label{solvlintheorem}
	Let $G$ be a connected, triangular by compact, solvable Lie group. Then $G$ is Lie-isomorphic to a definably linear group.
	\end{restatable}

\medskip

This strengthening is key for the main result of \Cref{sec3} (and one of the two main results of this paper): A complete characterization of when a linear Lie group is Lie isomorphic to a definable (definably linear) group.

	\begin{restatable}{theorem}{lintheorem}\label{criterionthm}
	Let $G$ be a connected linear Lie group whose solvable radical is triangular by compact. Then $G$ is Lie-isomorphic to a definably linear group.
	\end{restatable}

It maybe worth mentioning that we cannot avoid working up to Lie isomorphism in this theorem as there are some presentations of groups that are not definable in any $o$-minimal expansion of the reals. For example,
	
	\[ G= \left\{ \begin{pmatrix}
	e^t & 0 & 0\\
	0 & \cos (t) & -\sin (t) \\
	0 & \sin (t) & \cos (t)
	\end{pmatrix}  :t \in \mathbb{R} \right\} \]
	is not definable in any $o$-minimal expansion of $\mathbb{R}$ but it is Lie isomorphic to $(\mathbb{R},+)$ which is definable.

\medskip

We then work towards understanding the situation without the hypothesis of linearity. Namely, when is a connected Lie group $G$ Lie-isomorphic to a definable group?

Any definable group $G$ has a definable solvable radical $R$ (\cite{BJOcom}) and, as we will show in \Cref{solvlintheorem}, $R$ is in fact linear. Intending on using the Levi decomposition, we turn our focus to the Levi subgroups $S$ of $G$. If $S$ is definable it has finite center (since otherwise we would have a definable infinite discrete subgroup of a definable group which contardicts o-minimality). We show in \Cref{sec5}  that if $S$ has finite center we can build a definable group which is Lie-isomorphic to the original group, thus achieving a definability criterion for some non linear Lie groups: those that have a ``good Levi decomposition''.
	
	\begin{restatable}{theorem}{generalcase}\label{generalcase}
		Let $G$ be a connected Lie group. Let $R$ be its solvable radical and $S$ a Levi subgroup of $G$. If  $S$ has finite center and $R$ is triangular by compact then $G$ is Lie-isomorphic to a definable group.
	\end{restatable}

\medskip	

To prove this result we needed to work out a definability of finite covers of definable Lie groups, a result that may be of independent interest, so we included this result in  \Cref{sec4}.

\medskip

In order to have a full characterization of those Lie group that have a Lie-isomorphic definable copy, one would need to understand  definable groups without a good Levi decomposition. This work was started in \cite{levidecomin}. We examples and results in these papers (and what they mean for a complete characterization of which Lie groups have a definable copy) at the end of \Cref{sec5}.

\section{Definable linearity of definable solvable Lie groups}\label{sec2}

	In this section, we will prove \Cref{solvlintheorem}, a strengthening of the results of \cite{COSsolvable} stating that a solvable Lie group satisfying the conditions in the statement of \Cref{solvablecriterion} is actually Lie isomorphic to a definably linear group.

\medskip
	
	The following theorem of Malcev already tells us that a group satisfying \Cref{solvablecriterion} is actually \emph{linear} (it has a continuous faithful finite dimensional representation).
	
	\begin{fact}[\cite{Vinberg}, Theorem $7.1$]\label{semidirectproperty}
		Let $R$ be a solvable Lie group. Then $R$ admits a faithful finite-dimensional representation if and only if $R$ can be decomposed as a semidirect product $T\rtimes K$ where $T$ is simply connected and $K$ is a torus (maximal compact group).
	\end{fact}
	
	The general approach is to prove \Cref{extlemma}, the ``\nameref{extlemma},'' and use it to extend definably a definable representation of the supersolvable part $T$ of $G$ which can be obtained using triangularity of  $T$. This is essentially an adaptation of the proof of the analogue result in the real Lie group context. Particularly, our proof of the Extension Lemma essentially adapts the proof in the last chapter of \cite{Hochschild} to the definable context.

\medskip
		
	Recall that if $\rho : G \rightarrow GL(V)$ is a representation of $G$ in a finite-dimensional vector space $V$ and $(0)=V_k \leq\dots\leq V_0=V$ is a composition series for the $G$-module $V$, then the direct sum of the $V_i/V_{i+1}$ is a semi-simple $G$-module $V'$ (two such $G$-modules differ by a $G$-isomorphism by \emph{Jordan-Hölder's Theorem}).

We denote by $\rho '$ the representation associated to $\rho$ over $V'$ and we say that $\rho$ is \emph{unipotent} if $\rho '$ is trivial (which is equivalent to $N=\{\rho (g)-\operatorname{Id}_V : g\in G\}$ being nilpotent).

Let us fix a group $G$. We will consider two objects that are closely related. First, the space of representative functions of $G$ defined as:
	\[\mathcal{R}(G)=\{ f \in \mathcal{C}^0(G) : \dim \big( \text{span} ( \{ g\cdot f \}_{g\in G} )\big) <\omega \}\]
	where $\mathcal{C}^0(G)$ is the space of continuous functions $f:G\rightarrow \mathbb{R}$.
	When the group $G$ is linear we can fix a representation $\rho$ of $G$ and define a second space of representative functions \textbf{associated to $\rho$} as the space
	\[S(\rho )=\{ \varphi \circ \rho : \varphi \in End(V)^{\star} \}.\]
Notice that $S(\rho )\subseteq \mathcal{R}(G)$.
	
\medskip

We need the following lemma to ensure that the representation we build (in the proof of \Cref{solvlintheorem}) is in a finite dimensional vector space.
	
	\begin{fact}[{\cite[Chap.~XVIII, Lemma~2.1]{Hochschild}}] \label{lemmafindim}
		Let $G$ be a solvable Lie group and $\rho : G \rightarrow GL(V)$ be a faithful continuous representation in the finite-dimensional $\mathbb{R}$-vector space $V$. Let $A$ be a set of automorphisms of $G$ such that $\rho ' (\alpha (x) x^{-1})=\operatorname{Id}$ for all $x\in G$ and $\alpha \in A$. Then if $f\in S(\rho )$ is a representative function associated to $\rho$ then the vector space generated by $\{f\circ \alpha \}_{\alpha \in A}$ is finite-dimensional.
	\end{fact}
	
	The following lemma is the main brick in the proof of \Cref{solvlintheorem}.
	
	\begin{lemma}[Extension Lemma]
		\label{extlemma}
		Let $G=K\ltimes H$ be a group definable in an $o$-minimal expansion of the reals with $H$ and $K$ definable subgroups, $H$ normal and solvable. Suppose that $H$ admits a  faithful definable representation $\rho : H \rightarrow GL(V)$ where $V$ is a finite-dimensional vector space over $\mathbb{R}$. Suppose moreover that $\rho$ satisfies that for all $x\in G$ and $y\in H$ we have $\rho '(xyx^{-1}y^{-1})=id$. Then there is a definable representation $\sigma$ of $G$ that is faithful on $H$ and extends $\rho$.
	\end{lemma}
	
	\begin{proof} Let us consider the following action of $H$ on the space $\mathcal{C}^0(H)$ of continuous functions from $H$ to $\mathbb{R}$, for $h\in H , f\in \mathcal{C}^0(H)$ we define $h\cdot f : x\mapsto f(xh)$.
		This action is faithful.
		
		We can extend the action to $G$ as follows: For any $ h \in H,\ k \in K$ with $g=kh$ and $f\in \mathcal{C}^0(H)$ define
		\[g\cdot f : x\mapsto f(k^{-1}xkh). \]
		
		Notice that we had to consider the right action of $K$ on $H$ to get a left action of $G$ on $\mathcal{C}^0(H)$. Extending the action to $G$ preserves faithfulness on $H$.
		
		$H$ acts on  its space of representative functions $\mathcal{R}(H)$. We are going to show that not only $H$ but all of $G$ acts on $\mathcal{R}(H)$. Indeed, for $f\in \mathcal{R}(H)$, $h'\in H$ and $g=kh\in G$ with $h\in H$ and $k\in K$ then
	\[ h'\cdot (g \cdot f)= (h'kh)\cdot f=(kk^{-1}h'kh))\cdot f=k\cdot ((k^{-1}h'kh)\cdot f) \in k\cdot (H\cdot f).\]
As $(H\cdot f)$ is finite-dimensional, $hk\cdot f \in \mathcal{R}(H)$.
		
\medskip

		We now use \Cref{lemmafindim} to find a finite dimensional subspace $U$ of $\mathcal{R}(H)$ on which $H$ acts faithfully. Let $A$ be the set of automorphisms $c_k:x\mapsto k^{-1}xk$  of $H$ given by the conjugations in $H$ by $k\in K$. For all $f\in \mathcal{C}^0(H)$ we have $k\cdot f=f \circ c_k$; fix $f\in S(\rho )$. If we take $H$ instead of $G$, all of the hypothesis of \Cref{lemmafindim} are fulfilled, so that the vector space generated by the $ k\cdot f=f\circ c_k $ for $k \in K$ is finite-dimensional. But we know that $S(\rho )$ is finite dimensional (it has the same dimension as $End(V)^{\star}$), so the vector subspace $U\leq \mathcal{R} (H)$ generated by $G\cdot S(\rho )$ has finite dimension. $G$ acts on this space $U$, and as noticed earlier, the restriction of the action of $H$ is faithful. If $(f_1, f_2, \dots , f_k)$ is a basis for $U$, by definition of $U$ each $f_i=k_i \cdot (\varphi_i\circ\rho)$ for some $k_i\in K$ and $\varphi_i \in End(V)^{\star}$. Since these functions are definable and $\rho$ is also definable we get a finite-dimensional definable representation of $G$. \end{proof}

\medskip
	
	We now apply this \nameref{extlemma} to the decomposition of a solvable Lie group into its supersolvable and compact parts. In order to do so we need to find a representation of the supersolvable part that is unipotent (by which we mean upper-triangular with $1$'s on the diagonal).
	
	\begin{proposition}\label{unipotonnilp}\label{supersolvablerep}
		Let $G$ be a simply connected, connected, supersolvable Lie group and $N$ its nilradical (i.e.~its maximal normal nilpotent subgroup). Then $G$ is Lie isomorphic to a definably linear group $G_1$ whose nilradical $N_1$ is unipotent.
	\end{proposition}
	
	\begin{proof}	We first notice that since $G$ is simply connected and solvable, the exponential map gives us a diffeomorphism between $G$ and its Lie algebra $\mathfrak{g}$. $G$ is supersolvable, so $\mathfrak{g}$ is also supersolvable, and supersolvable Lie algebras have upper triangular representations (see \cite[Lemma $3.1$]{COSsolvable}). In any such representation $\mathfrak{n}$ is an upper triangular nilpotent subalgebra, so it must be strictly upper triangular. Then the matrix image of the exponential will be a linear Lie group $G_1$ Lie-isomorphic to $G$ whose nilradical $N_1=\exp (\mathfrak{n})$ is unipotent, as required. \end{proof}
	
	Recall the following:
	
		\begin{fact}[\cite{Vinberg}, Theorem $3.4$] \label{solvabletorsionfree}
			Let $G$ be a connected solvable Lie group of dimension $n$. The following are equivalent:
				\begin{itemize}
					\item $G$ is torsion-free,
					\item $G$ is simply connected,
					\item $G$ is diffeomorphic to $\mathbb{R}^n$.
				\end{itemize}
		\end{fact}

We are now ready to prove:
	
 	\solvlintheorem*
	
	\begin{proof} Let $G_1$ be the definable Lie group isomorphic to $G$ given by \Cref{solvablecriterion}, so that in particular $G_1$ is a  definable semidirect product of a supersolvable subgroup $H_1$ and a compact group $K_1$. Take the definable representation of $H_1$ given by \Cref{supersolvablerep} (it is simply connected because solvable and torsion-free as in \Cref{solvabletorsionfree}).

\medskip

In order to apply the \nameref{extlemma} we have to check that this representation satisfies the commutator condition, and this can be checked on its Lie algebra. Since $[\mathfrak{g}_1,\mathfrak{h}_1]$ is nilpotent (the commutator algebra of a solvable Lie algebra is nilpotent) it is included in $\mathfrak{n}=Lie(N)$ (where $N$ is the nilradical of $H_1$). But the representation we chose was unipotent on $N$ so it automatically satisfies the commutator condition. Applying the \nameref{extlemma} we get a definable representation $\rho$ of $G_1$ which is faithful on $H_1$.

\medskip

Let $\mu$ be any faithful continuous representation of $K_1$, which exists by the Peter-Weyl Theorem. Any faithful representation of a compact group is algebraic\footnote{This appears to be a well known fact but we could not find a reference, so present a proof in \autoref{appendixA}} so $\mu(K_1)$ is algebraic, and hence definable.

\medskip

The direct sum of $\rho$ and $\mu$ will be a definable and faithful representation of $G_1$, as required.
\end{proof}

	\section{From linearity to definability}\label{sec3}
	
	The main idea behind our result that extends the solvable case to the general linear case is to use the so-called \emph{Levi decomposition} of a linear Lie group $G$:
	
	\begin{fact}[Levi Decomposition, \cite{Levidec}, Theorem $1$]\label{levidecomposition}
		Let $G$ be a connected Lie group and $R$ its solvable radical (i.e.~its maximal solvable normal and connected Lie subgroup). There is a unique (up to conjugacy) maximal connected and semisimple subgroup $S$ of $G$ such that $G=RS$ and $\text{dim}(R\cap S)=0$.
		If the center $Z$ of $S$ is finite (which is the case whenever $G$ is linear) or if $G$ is simply connected then $G$ is an almost semidirect product of $R$ and $S$: $G=R(\rtimes)S$.
	\end{fact}
	
	We remind the reader of the definition of almost semidirect product. Let $G$ be a group, $H\unlhd G$ a normal subgroup and $K\leq G$ a subgroup. Then $G$ is an almost semidirect product of $H$ and $K$ if $G=HK$ and $H\cap K$ is finite.

All the results on the solvable case will articulate well with the rest of the group since the definable solvable radical and solvable radical coincide for definable groups:
	
		\begin{fact}[\cite{BJOcom}, Lemma $4.5$]
		Let $G$ be a definable group and $R(G)$ the group generated by all the normal and solvable subgroups of $G$ (not only the definable ones). Then $R(G)$ is a normal, definable and solvable subgroup of $G$.
		\end{fact}
	
In the linear case, any semisimple subgroup $S$ (in particular the subgroup $S$ in the Levi Decomposition of any linear group $G$) is semialgebraic:
	
		\begin{fact}[\cite{PPSlinear}, Remark $4.4$] \label{semisimplelinear}
		If $S$ is a connected semisimple linear Lie group, then $S$ is semialgebraic (i.e.~definable in $(\mathbb{R},0,+,\cdot ,<)$).
		\end{fact}

We give a last fact that will be put to use in the last part of the proof of the main theorem of this section, a particular case of Lema $3.1$ in \cite{Hochschild}.

\begin{fact}\label{quotientrep}
Let $G$ be a Lie group with a faithful representation $\rho$ on a finite dimensional $\mathbb R$-vector space $V$. Let $F$ be a finite central subgroup. Then there exists a real  finite dimensional vector space $W$ (a direct sum of finite tensor products of $G$-stable subspaces of $V$) such that the natural action of $G$ on $W$ has kernel $F$.
\end{fact}

Notice that $W$ is definable (any finite dimensional vector space is) and definability of if $\rho$ and $V$ will imply defimnability of the natural action of $G$ on $W$ (since it is a finite sum of tensor products of restrictions of $\rho$). In particular, quotients of linear definable groups by finite central kernels are linearly definable.

\medskip

We now give the proof of \Cref{criterionthm}. Throughout the proof we will repeatedly use the Lie functor that associates the Lie algebra to a Lie group.
	
	\lintheorem*

\begin{proof}
	The solvable radical $R$ is, by \Cref{solvlintheorem}, Lie-isomorphic to a definably linear group $R_1$. We would like to invoke the \nameref{extlemma} but \emph{a priori} the decomposition of $G$ into its solvable radical $R$ and a Levi subgroup $S$ is not fine enough. We will need to refine the decomposition. Consider the decomposition of $R$ into a simply connected (torsion free) normal subgroup $T$ and a compact subgroup $K$ as in \cite{COSsolvable}.
	
	We first show that $S$ can be chosen so that $S$ and $K$ commute making $SK$ a subgroup of $G$. Consider the adjoint action of $K$ on $\mathfrak{g}$ : we can see $\mathfrak{g}$ as a semisimple $K$-module (since $K$ is compact). Looking at the simple case one can easily see that $\mathfrak{g}=(\left[ \mathfrak{k},\mathfrak{g} \right] )\oplus \mathfrak{z}_{\mathfrak{g}}(\mathfrak{k})$ where $\mathfrak{k}=Lie(K)$ and $\mathfrak{z}_{\mathfrak{g}}(\mathfrak{k})$ its center. Write the Levi decomposition for $\mathfrak{z}_{\mathfrak{g}}(\mathfrak{k})$ as a semidirect sum $\mathfrak{z}_{\mathfrak{g}}(\mathfrak{k})=\mathfrak{u}_R+\mathfrak{u}_S$ where $\mathfrak{u}_R$ is the solvable radical and $\mathfrak{u}_S$ is a maximal semisimple subalgebra. But $\mathfrak{r}+\mathfrak{u}_R$ is a solvable ideal of $\mathfrak{g}$ (here $\mathfrak{r}=Lie(R)$) and since $\mathfrak{r}$ is the maximal solvable ideal of $\mathfrak{g}$ we must have $\mathfrak{r}+\mathfrak{u}_R \subseteq \mathfrak{r}$ so $\mathfrak{u}_R \subseteq \mathfrak{r}$. So the Levi decomposition of $\mathfrak{g}$ can be written $\mathfrak{g}=\mathfrak{r}+\mathfrak{u}_S$ with $\mathfrak{u}_S$ is a maximal semisimple subalgebra. Let $S_1$ be the connected subgroup of $G$ corresponding to the Lie algebra  $\mathfrak{u}_S$. It is also a Levi factor (maximal semisimple subgroup) of $G$ so it is in fact a conjugate of $S$. Since $\left[\mathfrak{u}_S,\mathfrak{k}\right]=(0)$ we get that $S_1$ commutes with $K$. By setting $S=S_1$ we may assume that $S$ commutes with $K$ and $SK$ is a group.
	
	We will show that $G$ is a semidirect product of $T$ and $H:=SK$. We already know that $G=(SK)T$  so we are left to show that $H\cap T$ is trivial. Take an element $sk=t$ in the intersection with $s\in S$, $k\in K$ and $t\in T$. Then $s=tk^{-1}\in S\cap R$ which is finite. So we have finitely many choices for $s$ and since $T\cap K$ is trivial, any fixed $s$ determines exactly the possible choices of the pair $(t,k)$. $H\cap T$ is therefore finite; but $T$ has no torsion so the intersection is trivial.

\medskip
	
	We will now use \Cref{unipotonnilp} to start with a representation $\rho$ of $T$ whose image is definable. We extend the representation using \cite[Chap.~XVII, Theorem~2.2]{Hochschild} (which is the ``Lie''-version of the \Cref{extlemma}) and we obtain a representation $\widetilde{\rho}$ of $G$ that is faithful on $T$ and such that the image of $T$ is definable.
	
But both $S$ and $K$ are semialgebraic (by \Cref{semisimplelinear} and  \Cref{compactlinear}), so the full image of $\widetilde{\rho}$ is definable.

\medskip
	
	To complete the proof we need a faithful and definable representation of $H=SK$. $S$ and $K$ commute so we have $H\simeq (S\times K)/F$ where $F=S\cap K$ is central in $H$ and finite (because $F$ is connected and $S$ is linear). Since both $K$ and $S$ have definable and faithful representations $\rho_K$ and $\rho_S$ respectively, $\sigma=\rho_K+\rho_S$ is a definable and faithful representation of $S\times K$. Using Fact \ref{quotientrep} we get a definable representation of $K\times S$ on a definable space $W$, with kernel $F$. This is a definable faithful representation $\widetilde{\sigma}$ of $H$.

Now $G=T\rtimes H$ so $H$ is Lie-isomorphic to $G/T$ and we can use the quotient map to get from $\widetilde{\sigma}$ to a definable representation of $G$ that is faithful on $H$; abusing notation, we sill also refer to this representation as $\widetilde{\sigma}$. Taking the direct sum of $\widetilde{\rho}$ and $\widetilde{\sigma}$ we get a faithful representation of $G$ whose image is definable.
\end{proof}

Since any definable linear Lie group is in fact Lie-isomorphic to a linearly definable group, we get the following:

\begin{theorem}
	Let $G$ be a linear Lie group, the following are equivalent:
	\begin{itemize}
		\item $G$ is Lie isomorphic to a group definable in an $o$-minimal expansion of $(\mathbb{R},0,+,1,\cdot,<)$.
		\item $G$ is Lie isomorphic to a group linearly definable in an $o$-minimal expansion of $(\mathbb{R},0,+,1,\cdot,<)$.
		\item $G$ is Lie isomorphic to a group definable in $\mathbb{R}_{an,exp}$.
		\item $G$ is Lie isomorphic to a group definable in $\mathbb{R}_{exp}$.		
	\end{itemize}
\end{theorem}

\begin{proof}
We will prove that any definable linear Lie group is actually definable in $\mathbb{R}_{exp}$. To see this take a linear Lie group $G$ satisfying the first condition. By \Cref{criterionthm} we get a matrix group $G_1$ that is definable and Lie-isomorphic to $G$. If we take in the proof of \Cref{criterionthm} we obtained linearity and definability analysing $G_1$ into a semidirect product $(SK)\ltimes T$ where $S$ was semialgebraic (it is semisimple and linear), $K$ algebraic (it is compact and linear) and $T$ definable in $\mathbb{R}_{exp}$ (it is the matricial exponential of an upper-triangular matrix). So $G_1$ as a set is definable in $\mathbb{R}_{exp}$; since the group law is given by matrix multiplication, it is polynomial in its coordinates, hence also definable in $\mathbb{R}_{exp}$. We notice that once we achieve linearity of the solvable radical $R=T\rtimes K$ of $G$ we do not need the analytic functions which were needed in \cite{COSsolvable}, since the compact part $K$ acts by matrix multiplication on the supersolvable part $T$.
	
	The other implications are implied by  \Cref{criterionthm} and results in \cite{COSsolvable}.
\end{proof}

\section{Finite covers of definable Lie groups}\label{sec4}

In order to generalize the previous theorem to the non linear case we will need to study finite covers of definable Lie groups. This is of independent interest, so we include these results in a separate section. We will make strong use of the construction of the $o$-minimal universal cover defined in \cite{universalcover}. We refer the reader to this paper for details and only use some of the facts proved there. In particular we will need the following.

\begin{theorem}\label{paths}
		Let $G$ be a connected definable Lie group. The Lie universal cover $\widetilde{G}$ and the $o$-minimal universal cover $\widetilde{G\mkern 0mu}^{\text{def}}\!\!\!\!$ of $G$ are Lie-isomorphic. Moreover there is a locally definable covering map $\pi_{\text{def}}: \widetilde{G\mkern 0mu}^{\text{def}}\longrightarrow G$ such that the following diagram commutes:
		\end{theorem}
		\[
		\begin{tikzcd}[column sep=large, row sep=large]
		\widetilde{G} \ar[d, "\pi"', twoheadrightarrow] \ar[r, "\sim"] & \widetilde{G\mkern 0mu}^{\text{def}}  \ar[dl, twoheadrightarrow, "\pi_{\text{def}}"] \\
		G
		\end{tikzcd}
		\]
		\begin{itshape}
			where $\pi:\widetilde{G}\twoheadrightarrow G$ is the usual Lie covering map.
		\end{itshape}

\begin{proof} The construction of $\widetilde{G\mkern 0mu}^{\text{def}}$ is based on the standard construction of $\widetilde{G}$ via  continuous paths (quotienting by homotopy), requiring the paths and homotopies to be both definable and continuous. The isomorphism from $\widetilde{G\mkern 0mu}^{\text{def}}$ to $\widetilde{G}$ will send a definable path to its homotopy class. We will need to show it is well defined (two definable maps which are homotopically equivalent have a definable homotopy between them) and that it is surjective, so that any path is homotopically equivalent to a definable one. We will prove the latter and skecth the proof of the former, which is essentially the same idea but the notation is much more complicated.

\medskip
		
$G$ is a definable Lie group we so can find (using cell decomposition) definable open sets $\mathcal{U}_1,\dots,\mathcal{U}_k$ covering $G$, each definably homeomorphic to an open convex, simply connected subset $\mathcal{B}_i$ in some cartesian power $\mathbb{R}^\ell$.

Take a continuous path $\sigma$ in $G$, and let $\sigma$ be parametrized with $t\in I:=\left[0,1\right]$.

By the Lebesgue's Number Theorem, there is some $\delta>0$ such that any subinterval of $I$ of size less than $\delta$ is contained in some of the $\mathcal{U}_i$'s. Let $m$ be such that $\frac{ 1}{m}<\delta$ and so that if $I_j:=\left[\frac{j }{m},\frac{(j+1)   }{m}\right]$, then the image of $I_j$ under $\sigma$ is fully contained in $\mathcal{U}_{i_j}$. We find a definable $\sigma_{\text{def}}$ as follows. Let $\mu_j$ be the push forward of $\sigma_j$ in $B_n$. By convexity, the straight line $l_j$ joining $\mu_j(\frac{j   }{m})$ and $\mu_j(\frac{(j+1)   }{m})$ is fully contained in $B_n$, and because $B_n$ is simply connected $\mu_j$ is homotopic to $l_j$. We then pull back $l_j$ to $U_n$ and the resulting path is $\sigma^{\text{def}}_j$. By construction, the starting and endpoints of $\sigma^{\text{def}}_j$ are the precisely the same as those of $\sigma_j$ which by definition define the connected path $\sigma$. The gluing of the $\sigma_j^{\text{def}}$ defines a definable connected path homotopic to $\sigma$, as required.

\medskip

The proof of the fact that two definable maps which are homotopically equivalent are definably homotopically equivalent is essentially the same: Let $F:[0,1]^2\mapsto G$ be an homotopy between definable paths $F(0,t)$ and $F(1,t)$. One uses the Lebesgue number theorem to find $m$ such that for any $i,j\leq m$ such that the image
\[F\left(\left[\frac{i   }{m},\frac{(i+1)   }{m}\right]\times \left[\frac{j   }{m},\frac{(j+1)   }{m}\right]\right)\] is fully contained in some $\mathcal{U}_{k}$.

As before, the pullback in $U_k$ of the straight line segment $l_{i,j}$ in $B_k$ joining the images of $F(\frac{i   }{m}, \frac{j   }{m})$ and $F(\frac{i   }{m}, \frac{j+1   }{m})$ is homotopically equivalent to $F\left(\frac{i   }{m}, \left[ \frac{j }{m}, \frac{j+1   }{m}\right]\right)$, and the same thing happens when we replace $i$ with $i+1$. But $l_{i,j}$ and $l_{i+1,j}$ are belong to the convex subset $B_k$ of $\mathbb R^\ell$, so they are definably equivalent via the straight-line homotopy. Taking the images in $U_k$, and doing this for every $i,j$ we build a definable homotopy equivalent to $F$.

\medskip

This concludes the proof.\end{proof}

Recall that if $G$ is a connected Lie group a covering map is a continuous surjective map $p:H\longrightarrow G$ where $H$ is a connected Lie group such that for each point $g\in G$ there is an open neighborhood $\mathcal{U}$ of $g$ such that $p^{-1}(\mathcal{U})$ is a disjoint union of open sets. We say that the covering is finite if $p$ has finite kernel. We have the following theorem.

\begin{theorem}\label{Liedefuniversal}
	Let $H\longrightarrow G$ be a finite covering map of a connected definable Lie group $G$. Then $H$ is Lie-isomorphic a definable Lie group $H_{\text{def}}$ and there is a continuous definable covering map $p_{\text{def}}: H_{\text{def}}\longrightarrow G$ such that the following diagram commutes.
		\[
	\begin{tikzcd}[column sep=large, row sep=large]
	H \ar[d, "p"', twoheadrightarrow] \ar[r, "\sim"] & H_{\text{def}}  \ar[dl, twoheadrightarrow, "p_{\text{def}}"] \\
	G
	\end{tikzcd}
	\]
\end{theorem}

\begin{proof}
	First, since $H$ is a finite cover of $G$ we have a covering map $\pi:H\twoheadrightarrow G$ with finite kernel. Now consider the universal cover $\widetilde{G}$ of $G$ and the definable universal cover $\widetilde{G\mkern 0mu}^{\text{def}}\!\!\!$ of $G$. By Theorem \ref{paths}, the two coincide.

From now on we will keep the notation $\widetilde{G}$ when we talk about the locally definable cover of $G$.

The universal property of universal covers gives us the following commutating diagram:
	
	\[
	\begin{tikzcd}[column sep=large, row sep=large]
	\widetilde{G} \ar[d, "\pi_G"', twoheadrightarrow] \ar[r, "\pi_H", twoheadrightarrow] & H  \ar[dl, twoheadrightarrow, "\pi"] \\
	G
	\end{tikzcd}
	\]
	
We know that $\widetilde{G}$ is a locally definable group and that $\pi_G$ is a continuous locally definable morphism. The morphisms $\pi_H$ and $\pi$ are only continuous but we are going to find a definable version of $H$ and the corresponding morphisms will also be definable. Since $\pi_G$ is locally definable and surjective we can find using (logic) compactness a definable subset $D_G$ in $\widetilde{G}$ such that the restriction $\pi_G:D_G \twoheadrightarrow G$ is surjective. We may of course assume that $D_G$ contains the identity.
	
Now $\text{Ker}(\pi)=\{h_1,\dots ,h_k\}$ is finite with $h_1=e_H$. We construct a definable set $D:=\bigcup_{i=1}^k  D_G \cdot g_i^{-1}$ where $g_i$ is a preimage of $h_i$ by $\pi_H$, choosing $e_G$ as $g_1$. Notice that the restriction of $\pi_H$ to $D$ is surjective.
	
Since $\text{Ker}(\pi_H)\subseteq\text{Ker}(\pi_G)$ and the latter is locally definable and discrete, any intersection with a definable set must be finite. This implies the following claim which we will use several times.

	\begin{claim}
		Let $X$ be a definable subset of $\widetilde{G}$ then $X\cap \text{Ker}(\pi_H)$ is a finite set.
	\end{claim}

		We can now define a group which is Lie-isomorphic to $H$. The universe is defined by taking $H_{\text{def}}:= D/\sim$ where $d_1\sim d_2$ if $\pi_H(d_1)=\pi_H(d_2)$ (equivalently, $d_1\cdot d_2\in D\cdot D^{-1} \cap \text{Ker}(\pi_H)$). This is a finite equivalence relation (and therefore definable) by the previous claim.
	
	We will define group multiplication and inverse operations on $D$ and check that they are compatible with $\sim$ and therefore pass nicely to $H_{\text{def}}$.
	
	Define $d_1\cdot d_2=d_3$ whenever $\pi_H(d_1\cdot d_2)=\pi_H(d_3)$. This is equivalent to $d_1\cdot d_2\cdot d_3^{-1}\in D^2\cdot D^{-1}\cap\text{Ker}(\pi_H)$, so it is a definable relation. Finally, we say that $d_1^{-1}= d_2$ if $d_1\cdot d_2\in D^2\cap\text{Ker}(\pi_H)$, so it is definable.
	
	This defines a group structure on $H_{\text{def}}$, and it also inherits the Lie group structure from $\widetilde{G}$ which, as it is shown in \cite{universalcover2}, it coincides with the usual $\tau$-topology on definable groups. Moreover, everything was defined so that the pushforward of $\pi_H$ to the $\sim$-quotient $H_{\text{def}}\rightarrow H$ is a continuous group isomorphism. So $H_{\text{def}}$ is a definable Lie group Lie-isomorphic to $H$, as required.
	\end{proof}

We get a nice corollary for semisimple Lie groups:
	
	\begin{corollary}
		Let $S$ be a connected semisimple Lie group. Then $S$ is Lie-isomorphic to a definable Lie group if and only if it has finite center.
	\end{corollary}
	
	This is simply because since $S$ has finite center, the quotien $S/Z(S)$ is centerless hence linear and semialgebraic. The map $\pi: S \rightarrow S/Z(S)$ is a finite covering of $S$.
	
	It follows that the definable Lie groups with definable Levi subgroups are precisely those with finite center Levi subgroups (this should probably be generalized to Lie groups over arbitrary real closed field with the appropriate covering theory).

\section{A Non linear case}\label{sec5}

In this section we take a look at the case when we don't assume linearity of the Lie group. As in the linear case, our analysis will be based on the components of the Levi-decomposition. We prove the following.

\generalcase*

\begin{proof}
 	Let $G$ be a group satisfying the hypothesis. By \cite{COSsolvable} and the linear case we may assume that the solvable radical $R$ is Lie isomorphic to $R_{\text{def}}$, a definably linear solvable group with the additional property that if we decompose $R_{\text{def}}=T\rtimes K\subseteq GL_n(\mathbb{R})$ with $T$ is supersolvable and $K$ compact, then $\mathfrak{r}=\text{Lie}(R_{\text{def}})=\mathfrak{t} + \mathfrak{k}\subseteq \mathfrak{gl}_n(\mathbb{R})$ and $\mathfrak{t}$ is upper triangular. Since $S$ has finite center we know by \Cref{levidecomposition} that $G$ is the almost semidirect product of $R$ and $S$. Specifically, there is a morphism $\varphi$ with
	
	\[	\begin{array}{lccl}
	\Psi: & R \rtimes_{\scriptscriptstyle\varphi} S & \rightarrow & G\\
	& (r,s) & \mapsto & r\cdot s
	\end{array}
	\]
	
	In order to build a definable version of $G$ we will need to construct a definable morphism $\varphi : S \rightarrow \text{Aut}(R_{\text{def}})$. To do so we need the ``appropriate'' version of $S$.

	Define the linearizer $\Lambda (G)$ of a connected lie group $G$ as the intersection of the kernels of all its continuous finite dimensional representations. It is a closed and normal subgroup and as the adjoint representation is continuous it is central. It is a theorem of Goto (\cite{Hochschildmostow}[Theorem 7.1]) that $G/\Lambda (G)$ has a faithful representation. That makes $\Lambda (G)$ the smallest normal closed subgroup $P$ such that $G/P$ is Lie isomorphic to a linear group.
	
	In our case $S$ has finite center and we do not need to invoke Goto's theorem, we just need the following lemma:
	
	\begin{lemma}
		Let $G$ be a Lie group and suppose that $G/H_1$ and $G/H_2$ (with $H_1$ and $H_2$ normal and closed subgroups) both have faithful representations $\rho_1$ and $\rho_2$. Then $G/(H_1\cap H_2)$ has a faithful representation.
	\end{lemma}
	
	\begin{proof}
		Let us write $\widetilde{\rho_i}:=\pi_i \circ \rho_i$ where $\pi_i: G \twoheadrightarrow G/H_i$ is the canocical projection. The direct sum  $\widetilde{\rho_1}\oplus\widetilde{\rho_2}$ is a representation of $G$ whose kernel is $H_1\cap H_2$.
	\end{proof}
		
	Since $S$ has finite center, $\Lambda(S)\subseteq Z(S)$, and $S/Z(S)$ is centerless and therefore linear, applying the lemma a finitely many times gives us a faithful representation of $S/\Lambda(S)$. Let $S_{\text{lin}}$ be a linear copy of $S/\Lambda (S)$.
	
	\medskip
	
	$S_{\text{lin}}$ is linear and semisimple, hence semialgebraic. $S$ is a finite cover of $S_{\text{lin}}$ so by Theorem \ref{Liedefuniversal} there is  a group $S_{\text{def}}$ that is definable, Lie isomorphic to $S$ and a definable surjection $\pi: S_{\text{def}} \twoheadrightarrow S_{\text{lin}}$. Given that $S_{\text{lin}}$ is the biggest linear quotient of $S$ any representation of $S$ must factor through $S_{\text{lin}}$ .
	
	\medskip
	
	Finally, consider the action of $S_{\text{def}}$ on $R_{\text{def}}$ induced by the isomorphism from $S$ to $S_{\text{def}}$. This is a morphism $\varphi': S_{\text{def}} \rightarrow \text{Aut}(R_{\text{def}}) \subseteq \text{Aut}(\mathfrak{r}_{\text{def}}) \subseteq \text{GL}(\mathfrak{r}_{\text{def}})$. This must therefore factor through $S_{\text{lin}}$ and we have the following commutating diagram:
		\[
	\begin{tikzcd}[column sep=large, row sep=large]
	S_{\text{def}} \ar[d, "\pi"', twoheadrightarrow] \ar[r, "\varphi"] & \text{Aut}(R_{\text{def}}) \ar[r, hook, "\text{Der}"] & \text{Aut}(\mathfrak{r})\subseteq GL(\mathfrak{r})\\
	S_{\text{lin}} \ar[ru,  "\varphi_{\text{lin}}"'] \ar[urr, bend right, "\varphi_{\text{Lie}}"']
	\end{tikzcd}
	\]
	
	The morphism $\varphi_{Lie}$ maps a semisimple linear Lie group to a linear group, its graph $\Gamma$ is Lie isomorphic to $S_{\text{lin}}$, so it
	must be semialgebraic (recall that any linear semisimple Lie group is semialgebraic by \Cref{semisimplelinear}).
	
	We will prove that the action of $S_{\text{lin}}$ on $R_{\text{def}}$ is definable. Because the maximal torsion-free subgroup $T$ of $R_{\text{def}}$ is characteristic it must be stable under the action. Recall that our construction of $T$ implies that the Lie algebra $\mathfrak{t}$ is supersolvable and that the exponential restricted to $\mathfrak{t}$ is definable. For any $s\in S_{\text{lin}}$ and $t\in T$, 	we can choose $X\in\mathfrak{t}$ such that $\exp(X)=t$ and we have
	\[\varphi_{\text{lin}}(s)(t)=\exp\restriction_\mathfrak{t}(\varphi_{Lie}(s)(X))\]
	which implies that the restriction $\varphi_{\text{lin}}^T$ of the action to $T$ is definable.
	
	\medskip
	
	Understanding definability of $\varphi_{\text{lin}}$ on the maximal compact subgroup is a little harder.
	
	\begin{claim}
		Let $K_0$ be a maximal compact subroup, $\mathcal{K}$ be the set of maximal compact subgroups of $R_{\text{def}}$ and $\mathbf{K}=R_{\text{def}}/T$. Then the following hold:
		
		\begin{enumerate}[label=(\roman*)]
			\item The natural action of $S_{\text{lin}}$ on $\mathbf{K}$ is trivial.
			\item $\mathcal{K}$ is equal to $\{tK_0t^{-1}:t\in T\}$ and there is a definable bijection between $\mathcal{K}$ and $T/_\sim$, a quotient of $T$ by a definable equivalence relation.
			\item $S_{\text{lin}}$ acts on $\mathcal{K}$ and the orbit of $K_0$ is definable.
			\item $S_\text{lin}$ acts definably on $K_0$.
		\end{enumerate}
	\end{claim}
	
	\begin{proof}
		\leavevmode
		\begin{enumerate}[label=(\roman*)]
			\item Any automorphism of $R_{\text{def}}$ must fix $T$ which implies that $S_{\text{lin}}$ acts by continuous automorphisms on $\mathbf{K}$ in a natural way; this can be seen as a continuous map $\overline\varphi$ from $S_{\text{lin}}$ to $\text{Aut}(\mathbf{K})$. Since $\mathbf{K}$ is a torus, $\text{Aut}(K)\simeq GL_m(\mathbb{Z})$ for some $m\in\mathbb{N}$ and it is discrete.\footnote{$K\simeq \mathbb{R}^m/\mathbb{Z}^m$ so any automorphism of $K$ is an automorphism of $\mathbb{R}^m$ that leave $\mathbb{Z}^m$ fixed.} But $S_{\text{lin}}$ is connected, so $\overline\varphi$ must be constant on $\text{Aut}(\mathbf{K})$, and the action is trivial.

			\item It is known that all maximal tori of a linear group are conjugate, so $\mathcal K$ is the set of conjugates by elements of $R_{\text{def}}$ of $K_0$. But $R_{\text{def}}=T\cdot K_0$ which implies that $\mathcal K=\{tK_0t^{-1}:t\in T\}$ and is therefore in bijection with the quotient $T/_\sim$ where		
			\[
			t_1\sim t_2 \Leftrightarrow t_1\cdot K_0\cdot t_1^{-1}=t_2\cdot K_0\cdot t_2^{-1}.
			\]		
			Since $K_0$ is a linear tori, it is algebraic and therefore definable; this means that $\sim$ is definable and so is $T/_\sim$ (elimination of imaginaries in $o$-minimal theories).
	
			\item The action of $S_{\text{lin}}$ on $\mathcal{K}$ can be represented by a map \[\Phi:S_{\text{lin}}\times T/_\sim \rightarrow T/_\sim\]
			and we will show that  the restriction of $\Phi$ to $S_{\text{lin}}\times \{[e_T]\}$ is definable where $[e_T]$ is the class of the identity in $T/_\sim$.
			
			We begin by picking an open neighborhood $\mathfrak{u_0}$ of $0$ in $\mathfrak{r}$ such that $\exp :\mathfrak{u_0}\rightarrow R_{\text{def}}$ is a diffeomorphism. If needed, we shrink $\mathfrak{u_0}$ such that we obtain a local diffeomorphism $\exp : \mathfrak{u} \rightarrow U$ with $\mathfrak{u}$ bounded and $U$ open neighborhood of the identity. This restriction is an analytic function defined on a compact subset of $\mathfrak{r}$, and is therefore definable in the o-minimal structure $\mathbb{R}_{\text{an},\exp}$.
			Now we claim that $\Phi(s,[e_T])=[t]$ if and only if
			\[\exp(\mathfrak{u}\cap \varphi_{\text{Lie}}(s)(\mathfrak{k_0}))=U\cap t K_0 t^{-1}\]
			where $\mathfrak{k_0}=\text{Lie}(K_0)$, and since the latter is definable, $\Phi(\cdot,[e_T])$ is a definable function.
			
		The equivalence follows because if we define $\mathfrak{k}_t:=\text{Lie}(K_t)$ with $K_t:=t\cdot K_0\cdot t^{-1}$, then by the Lie correspondance any two distinct maximal compact subgroups $K_{t_1}$ and $K_{t_2}$ correspond to different Lie sub-algebras $\mathfrak{k}_{t_1}$ and $\mathfrak{k}_{t_2}$ of $\mathfrak{r}$. But Lie subalgebras are vector subspaces and  $\mathfrak u$ is an open neighborhood of the identity, so $\mathfrak{k}_{t_1}\neq \mathfrak{k}_{t_2}$ if and only if $(\mathfrak{k}_{t_1}\cap \mathfrak u )\neq (\mathfrak{k}_{t_2} \cap \mathfrak u)$ as required.

			\item For any element $t\in T$ and $\overline k \in \mathbf{K}$ there is a unique element $k_t\in K_t$ such that $p(k_t)=\overline k$ where $p:R\rightarrow R/T$. Moreover the map $\sigma_t:\overline k \rightarrow k_t$ is definable.
			We can now define the action of $S_{\text{lin}}$ on $K_0$ by defining
			\[	\begin{array}{lccl}
			\varphi_{\text{lin}}^{K_0}: & S_{\text{lin}} \times K_0 & \longrightarrow & R_{\text{def}}\\
			& (s,k) & \mapsto & \sigma_{\Phi(s,[e_T])}(p(k)).
			\end{array}
			\]
		\end{enumerate}
	\end{proof}

	Definability of $\varphi_{\text{lin}}$ now follows easily. Let $s\in S_{\text{lin}}$ and $r\in R_{\text{def}}$. Then we can write $r=t\cdot k$ with $t\in T$ and $k\in K_0$. Then we get:
	
	\[\varphi_{\text{lin}}(s)(r):=\varphi_{\text{lin}}(s)(t)\cdot\varphi_{\text{lin}}(s)(k)=\varphi_{\text{lin}}^T(s, t)\cdot\varphi_{\text{lin}}^{K_0}(s, k),\]
	so it is definable in $\mathbb{R}_{\text{an},\exp}$. Definability of $\varphi=\varphi_{\text{lin}}\circ\pi$ follows.
	
\end{proof}

\subsection{Further remarks}

The conditions on $R$ in Theorem \ref{generalcase} are necessary. However, in \cite{Connectedcomp1} the authors exhibit
an example of a group definable in an o-minimal structure for which any Levi subgroup has infinite center, so we know
that this is not a necessary condition for a group $G$ to be definable.

In any group $G$ definable in an o-minimal expansion of $\mathbb R$ the solvable radical must be definable, which implies that
$G/R$ is isomorphic to a definable group (using elimination of imaginaries). It is also known that any Levi subgroup will be a cover of $G/R$, albeit maybe not a finite one (like in the example in \cite{Connectedcomp1}).

We do know, however, that $G$ is an extension of $G/R$ by $R$, so one could hope to give necessary and sufficient conditions understanding the
actions of $G/R$ on $R$ (maybe along the lines which we did in this paper), and then characterizing the definable 2-cocycles in $H^2(G/R, R)$ which give rise to definable groups which are not almost semidirect products of $G/R$ and $R$. This appears to be the case in the example in \cite{Connectedcomp1} and might be the path to improve the results in this paper to get necessary and sufficient conditions for a connected Lie groups $G$ to be definable.

\medskip

A different subject for further research is the concept of the linearizer of a Lie group, which appeared as a very useful tool in Lie theory in the proof of \Cref{generalcase}, the linearizer of a Lie group. Specifically, if $G$ is a group definable in an $o$-minimal expansion of a real closed field $\mathcal{R}$, and $\Lambda_{\text{def}}$ is the intersection of the kernel of all its definable representations in a finite dimensional $\mathcal{R}$-vector space, then we know that $\Lambda_{\text{def}}$ is a definable subgroup of $G$ by descending chain condition and central (adjoint representation). Is $\Lambda$ is the intersection of the kernel of all representations, namely, is $\Lambda_{\text{def}}=\Lambda$? Is $\Lambda$ definable?

	\appendix
	\titleformat{\section}[hang]{\large\bfseries}{Appendix \thesection:\ }{0pt}{}
	\newpage
	
	\section{Algebraicity of linear compact groups}\label{appendixA}
	
	We think the following is well-known but could not find any clear reference so we provide a proof.
	
	\begin{theorem}\label{compactlinear}
		Let $K$ be a connected compact subgroup of $GL_n(\mathbb{R})$. Then $K$ is definable in $(\mathbb{R},+,\cdot,0,1)$, more precisely it is an algebraic subgroup of $GL_n(\mathbb{R})$.
	\end{theorem}

	The standard statement one finds in the literature states that a compact real Lie group is Lie-isomorphic to an algebraic group. Although quite interesting, this statement does not even imply definability of a connected compact linear group because all we get is up to isomorphism. We will follow the sketch stated in \cite{Vinberg} which uses a classical analytic argument. We will need the following lemma.

	Any compact group $K$ admits a Haar measure which makes it amenable. Results in \cite[Chapter 12.4]{Wagon} yield the following.

	\begin{fact}\label{amenable}
	Let $K$ be a compact group acting linearly and continuously on a real vector space $V$, and let $\mu$ be the Haar measure of $K$. Let $v\in V$ such that the subspace $W:=\text{Span}(K\cdot v)$ is finite dimensional. Then
	\[
	\int_K k\cdot v \; d\mu
	\]
	is a fixed by $K$ and lays in $W$(in fact, it belongs to the convex Hull of $K\cdot v$).
	\end{fact}

\begin{proof}[Proof of Theorem \ref{compactlinear}.]

Let us consider the vector space of $n$-squared matrices $V:=\mathcal{M}_n(\mathbb{R})$ on which $K$ acts (linearly) by multiplication.
	
Let $\mathbb R\left[ V \right]$ be the algebra of polynomials of $V$ (if $\left(e_1,\dots ,e_{n^2} \right)$ is any basis of $V$ and $\left(e_1^*,\dots ,e_{n^2}^* \right)$ is the corresponding a dual basis then $\mathbb{R}\left[ V \right] := \mathbb{R}\left[ e_1^*,\dots , e_n^* \right] $). Since $K$ acts linearly on $V$ it also acts linearly on its polynomial algebra: for $g\in K$, $P\in \mathbb{R}\left[ V \right]$ and $v\in V$, $g\cdot P (v):= P(g^{-1}\cdot v)$.

For any $P\in \mathbb{R}\left[ V \right]$ the degree of $P$ stays bounded under the action. That implies that $\text{Span}(K\cdot P)$ is a finite dimensional $K$-invariant subspace of $\mathbb R\left[ V\right]$.
	
Let us go back to the action of $K$ on $\mathcal{M}_n(\mathbb{R})$, and let $\mathcal K_1$ and $\mathcal K_2$ be any two orbits. These are compact disjoint subsets so there is an analytic function $f$ whose value on on orbit is $0$ and $1$ on the other one. By \emph{Weierstrass approximation theorem}, for any $\varepsilon$, there is a polynomial $P$ in $\mathbb{R}\left[ \mathcal{M}_n(\mathbb{R}) \right]$ which is $\varepsilon$-close to $f$, which implies that $P(\mathcal K_0)\subseteq \left[-\varepsilon,\varepsilon\right]$ and $P(\mathcal K_1)\subseteq \left[1-\varepsilon,1+\varepsilon\right]$.

$K=\mathcal K_0$ itself is the orbit of the identity in $\mathcal{M}_n(\mathbb{R})$. For any other orbit $\mathcal K$, let $P_{\mathcal K}$ be any polynomial such that $P_{\mathcal K}(\mathcal K_0)\subseteq \left[-\varepsilon,\varepsilon\right]$ and $P_{\mathcal K}(\mathcal K)\subseteq \left[1-\varepsilon,1+\varepsilon\right]$ with $\varepsilon <1/2$. By  \Cref{amenable}, the element
\[\int_K k\cdot P_{\mathcal K}\;d\mu \] is an element of $\mathbb R[V]$ which is fixed by $K$. Because the value is constant at $K$, by subtracting a constant (which by construction is the Haar-measure average of the values of $k\cdot P_{\mathcal K}$ on $K$ so it must be less than $\epsilon$) we get a polynomial which when evaluated in $K$ gives 0 and different from 0 (in fact larger than $1-\epsilon$) for any element of $\mathcal K$.

Now let $\mathcal{I}$ be the ideal generated by all of these polynomials, where we vary $\mathcal K$ over all $K$-orbits in $\mathcal{M}_n(\mathbb{R})$ which are different from $K$. Since $\mathbb{R}\left[ \mathcal{M}_n(\mathbb{R}) \right]$ is Noetherian, $\mathcal{I}$ is finitely generated, say by $P_1,\dots ,P_h$. Let $X=(x_{i,j})\in\mathcal{M}_n(\mathbb{R})$, the system of equations in $x_{i,j}$ given by all of the coefficient equalities of $\{P_j(X)=0\}_{i=1}^h$ gives us a definition for $K$. Indeed, since the $P_i$'s are $K$-invariant, any element in $K$ satifies all these equations. On the other hand, if a matrix $M$ is not in $K$ then its orbit $\mathcal K$ is disjoint from $K$ and the polynomial $P_\mathcal K\in \mathcal {I}$ satisfies $P_{\mathcal K}(M)\neq 0$. This must be reflected in at least one of the generators. So
	
	\[K=\left\{ X\in \mathcal{M}_n(\mathbb{R}) : \; \forall \, 1\leq i \leq h \; P_i(X)=0   \right\}.\]
	
		Those equations are purely algebraic giving us the algebraic representation we were looking for.
\end{proof}

\pagebreak
\bibliographystyle{abbrv}
\bibliography{Paperlinealbiblio}
	
\end{document}